\documentclass[11pt]{article}

\usepackage{lineno}
\usepackage{amssymb,amsthm,amsmath}
\usepackage{amsfonts}
\usepackage{graphicx}
\usepackage[small]{caption}
\usepackage{subcaption}
\usepackage{epsfig}

\usepackage[utf8]{inputenc}
\usepackage{fullpage}
\usepackage{enumerate}
\usepackage{url}

\usepackage{hyperref}

\newtheorem{theorem}{Theorem}
\newtheorem{lemma}{Lemma}

\newtheorem{proposition}{Proposition}
\newtheorem{corollary}{Corollary}

\newtheorem{problem}{Problem}

\newcommand{\ie}{{i.e.}}
\newcommand{\eg}{{e.g.}}

\newcommand{\conv}{{\rm conv}}

\newcommand{\NN}{\mathbb{N}} %  set of natural numbers
\newcommand{\RR}{\mathbb{R}} %  set of real numbers
\def\T{\mathcal T}

\newcommand{\later}[1]{}
\newcommand{\old}[1]{}

\title{{\sc Peeling Sequences}}

\author{
Adrian Dumitrescu\thanks{%
Algoresearch L.L.C., Milwaukee, WI, USA. 
Email~\texttt{ad.dumitrescu@algoresearch.org}.}
\and
G\'eza  T\'oth\thanks{Alfred R\'enyi Institute of Mathematics,
Budapest, Hungary\@. Email: \texttt{geza@renyi.hu}.
Supported by National Research, Development and Innovation Office, NKFIH, K-131529 and  ERC Advanced Grant ``GeoScape,'' No. 882971.}
}

\begin{document}

\maketitle

\begin{abstract}
  Given a set of $n$ labeled points in general position in the plane,
  we remove all of its points one by one. At each step, one point from
  the convex hull of the remaining set is erased. 
  In how many ways can the process be carried out?
  The answer obviously depends on the point set. 
  If the points are in convex position, there are exactly $n!$ ways, which is the maximum number of ways
  for $n$ points. But what is the minimum number?
  It is shown that this number is (roughly) at least $3^n$ and at most $12.29^n$. 


\end{abstract}

\section{Introduction} \label{sec:intro}

A set of points in the plane is said to be in \emph{general position} if no three of them are collinear.
Let $P$ be a set of $n$ points in the plane in general position. 
Consider the following iterative process:
remove points one by one until no point remains,
under the provision that exactly one extreme point, that is,
a vertex of the convex hull is removed in each step.
If the points are in convex position, there are exactly $n!$ ways to do this,
which is clearly the maximum number of ways
for removing $n$ points. What is the minimum number? 

  A \emph{peeling sequence} for $P$ is any permutation of the points of $P$ that can be obtained by
  writing the labels of the points removed one by one. We are interested in the \emph{minimum}
  number of such permutations that can be obtained, over all point sets of size $n$. 

\paragraph{Definitions.}
Given a point set $P$ in general position the plane,
let $g(P)$ count the number of peeling sequences for $P$; and
$g(n)$ denote the minimum of $g(P)$ over all $n$-element point sets $P$ in general position.
It is easy to see that $g(n)$ is an increasing integer sequence; it is now
entry \texttt{A358251} in~\cite{Sl22}.
Observe that when $n \geq 3$, the last three points can be removed in any order; there are $3!=6$ ways,
whence $g(n)$ is a multiple of $6$ for every $n \geq 3$. 

In an earlier writing~\cite{Du22} the following bounds were obtained (all logarithms are in base $2$):
(i)~every $n$-element point set in general position admits $\Omega(3^n)$ peeling sequences;
(ii)~on the other hand, there are sets with $2^{O(n \log \log {n})}$ peeling sequences. 
Recently, a further improved upper bound, $2^{O(n \log \log \log {n})}$, was presented
by the first named  author~\cite{Du23}.
Here we significantly improve the upper bound; in particular, we show that it is of the
form $O(a^n)$ for some $a>1$.

The problem can be naturally generalized to point sets in higher dimensions.
A set of points in the $d$-dimensional space $\RR^d$ is said to be:
(i)~in \emph{general position} if any at most $d+1$ points are \emph{affinely independent}; and
(ii)~in \emph{convex position} if none of the points lies in the convex hull of the other points.
Let $d\ge 2$ and let $P$ be a point set in $\RR^d$ in general position.
We denote by $g(P)$  the number of peeling sequences of $P$ and let $g_d(n)$
be the minimum of $g(P)$ over all $n$-element point sets $P$ in general position in  $\RR^d$.
In particular, $g_2(n)$ is the same as $g(n)$.
For any fixed $d$ we obtain exponential upper and lower bounds for  $g_d(n)$.

\paragraph{Our results.}
We first observe that  every $n$-element point set has $\Omega(3^n)$ peeling sequences.
From the other direction, we show that if $n$ is a power of $3$, we can find suitable configurations
with a small number of peeling sequences.

\begin{theorem} \label{thm:simple}
Let $n=3^k$, for some $k \in \NN$. 
  There is a set of $n$  points in general position in the plane with at most $27^n$ peeling sequences.
\end{theorem}

\begin{corollary} \label{cor:simple}
  For every $n\in \NN$,  there is a set of $n$ points in
  general position in the plane with at most $19683^n$ peeling sequences.
\end{corollary}

We next arrive at our main result which summarizes our best lower and upper bound, respectively.

\begin{theorem} \label{thm:main}
 Every $n$-element point set in general position in the plane has $\Omega(3^n)$ peeling sequences.
 On the other hand, for every $n\ge 3$ there is a  point set in general position with at most
 $12.29^{n}/100$  peeling sequences.
\end{theorem}

The exponential lower bound from the planar case as well as the construction in the proof of
Theorem \ref{thm:simple} can be generalized to higher dimensions.

\begin{theorem} \label{thm:d-dim}
 Every $n$-element point set in general position in $\RR^d$ has $\Omega((d+1)^n)$ peeling sequences.
 On the other hand, if $n=(d+1)^k$, where $k \in \NN$,  there is a set of $n$ points in general position
 in $\RR^d$ with at most $(d+1)^{(d+1)n}$ peeling sequences.
\end{theorem}

\begin{corollary} \label{cor:d-dim}
  For every $n\in \NN$,  there is a set of $n$ points
  in general position in $\RR^d$ with at most
  $(d+1)^{(d+1)^2n}$ peeling sequences.
\end{corollary}

In this case we did not attempt to optimize the base of the exponential function.

\paragraph{Related work.}
The concept of ``peeling'' a convex hull has been associated to dynamic convex hull algorithms and
convex hull determination~\cite{OL81}.
For a planar point set $P$, the \emph{convex layers} of $P$ are the convex polygons obtained by iterating the
following procedure: compute its convex hull and remove its vertices from $P$~\cite{ANW21,Ch85}. 
The process of peeling a point set has been shown very useful in obtaining robust estimators
in statistics~\cite{Ch85,Sh76}.
A common estimator of a (unidimensional) sample is its arithmetic mean, however this estimator
can be severely affected by outliers.  A method that performs better is is to discard
the highest and lowest $\alpha$-fraction
of the data and take the mean of the remainder. This is the $\alpha$-\emph{trimmed mean}. 
The median is the special case $\alpha=1/2$. The higher dimensional analog of trimming, called ``peeling''
by Tukey, consists of successively removing extreme points of the convex hull of the data until a certain
fixed fraction of the points remains.

A quadratic algorithm for peeling (\ie, for convex layer decomposition) was initially proposed by Shamos~\cite{Sh78}.
A faster algorithm, running in $O(n \log^2 {n})$ is due to Overmars and Van Leeuwen~\cite{OL81}.
Finally, an optimal algorithm, running in $O(n \log {n})$ time was obtained by Chazelle~\cite{Ch85}. 
Computing the convex layers and studying their structure in a random setting have been studied by
Dalal~\cite{Da04}. Har-Peled and Lidick{\'{y}}~\cite{PL13} have studied the number of steps
needed for peeling the integer grid $G_n$ with $n$ points (with, say, $n=k^2$); note that $G_n$ is
\emph{not} in general position, however the peeling process can be executed on any point set. 

It should be noted that while in the discussion above all the extreme points in a layer are removed in
parallel (\ie, at the same time), in our study --- of the function $g(n)$ --- the extreme points are removed
sequentially one by one. 

After some preliminaries in Section~\ref{sec:prelim}, we prove our main results,
Theorems~\ref{thm:simple} and~\ref{thm:main},  in Section~\ref{sec:main}.
Theorem~\ref{thm:d-dim} regarding higher dimensions can be found in Section~\ref{sec:d-dim}.
We conclude with some remarks in Section~\ref{sec:remarks}.

\section{Lower bound and small values of $n$} \label{sec:prelim}

As a warm-up we determine the values of $g(\cdot)$ for the first few values of $n$.
Trivially, we have $g(1)=1$, and $g(2)=2$.

\begin{proposition} \label{prop:small} 
  The following exact values can be observed. 
\begin{align*}
g(3) &= 6, \\
g(4) &= 18, \\
g(5) &= 60, \\
g(6) &= 180.
\end{align*}
\end{proposition}
\begin{proof}
Let $h$ denote the size of the convex hull of $P$. 
  The case $n=3$ is clear: there are $3!=6$ permutations and each of them is a valid peeling sequence.
The remaining cases are illustrated in Fig.~\ref{fig:small}. 

\begin{figure}[ht]
\centering
\includegraphics[scale=0.99]{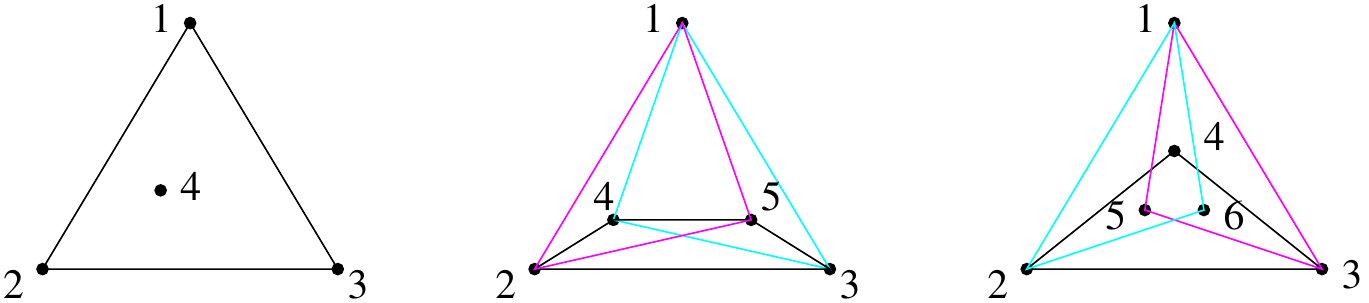}      
\caption{Illustration for $n=4,5,6$.}
\label{fig:small}
\end{figure}

  Let now $n=4$. If the points are in convex position, there are $4!=24$ permutations.
  If the points are not in convex position, let $4$ be the interior point. Then any permutation
  that starts with $4$ is invalid, however, all remaining $4!-3!=18$ permutations are valid.

  Let now $n=5$.  If $h \geq 4$, there are at least $4 \times g(4) =72$ permutations.
  The case $h=3$ yields the smallest number, $24+18+18=60$, of permutations; indeed,
  removing one of the extreme points yields a convex quadrilateral, while the other two
  removals can result in a triangle with a point inside. 

  Finally, let $n=6$.  If $h \geq 4$, there are at least $4 \times g(5) =240$ permutations.
  The case $h=3$ yields the smallest number: $3 \times g(5) = 3 \times 60=180$ permutations; indeed,
  removing each of the extreme points may yield the minimizer for $n=5$ discussed above. 
\end{proof}

\paragraph{Lower bound.}
It is clear that $g(2)=2$. Assume now that $n \geq 3$. 
Let $P$ be any $n$-element point set in general position. By the assumption, $\conv(P)$ has
at least three extreme vertices, removal of each yields a set of $n-1$ points in general position.
Any two peeling sequences resulting from the removal of two different extreme vertices are clearly different.
As such, $g(\cdot)$ satisfies the recurrence $g(n) \geq 3 \cdot g(n-1)$. Consequently, $g(n) = \Omega(3^n)$.
By Proposition~\ref{prop:small} we also have $g(n) \geq 180 \cdot 3^{n-6}$ for every $n \geq 6$.
\qed

\section{Upper bounds}\label{sec:main}

\subsection{First construction and analysis } \label{subsec:first}

Let $Z=X \cup Y$ be  a finite point set in general position, where $X \cap Y =\emptyset$. 
Consider any peeling sequence, say, $\pi$, of $Z$.
Then $\pi$ naturally induces two peeling sequences, one for $X$ and one for $Y$.
This implies, that $g$ is a monotone increasing function.

\paragraph{Proof of Theorem \ref{thm:simple}.}
We construct the sets $S_k$ recursively for $k \in \NN$,  
such that the set $S_k$ contains $n=3^k$ points and $g(S_k)\le 27^n$.
Moreover, $S_k$ is {\em flat}, that is, all points are very close
(compared to the minimum distance in $S_k$) to a line; see Fig.~\ref{fig:construction}.

\begin{figure}[ht]
\centering
\includegraphics[scale=0.6]{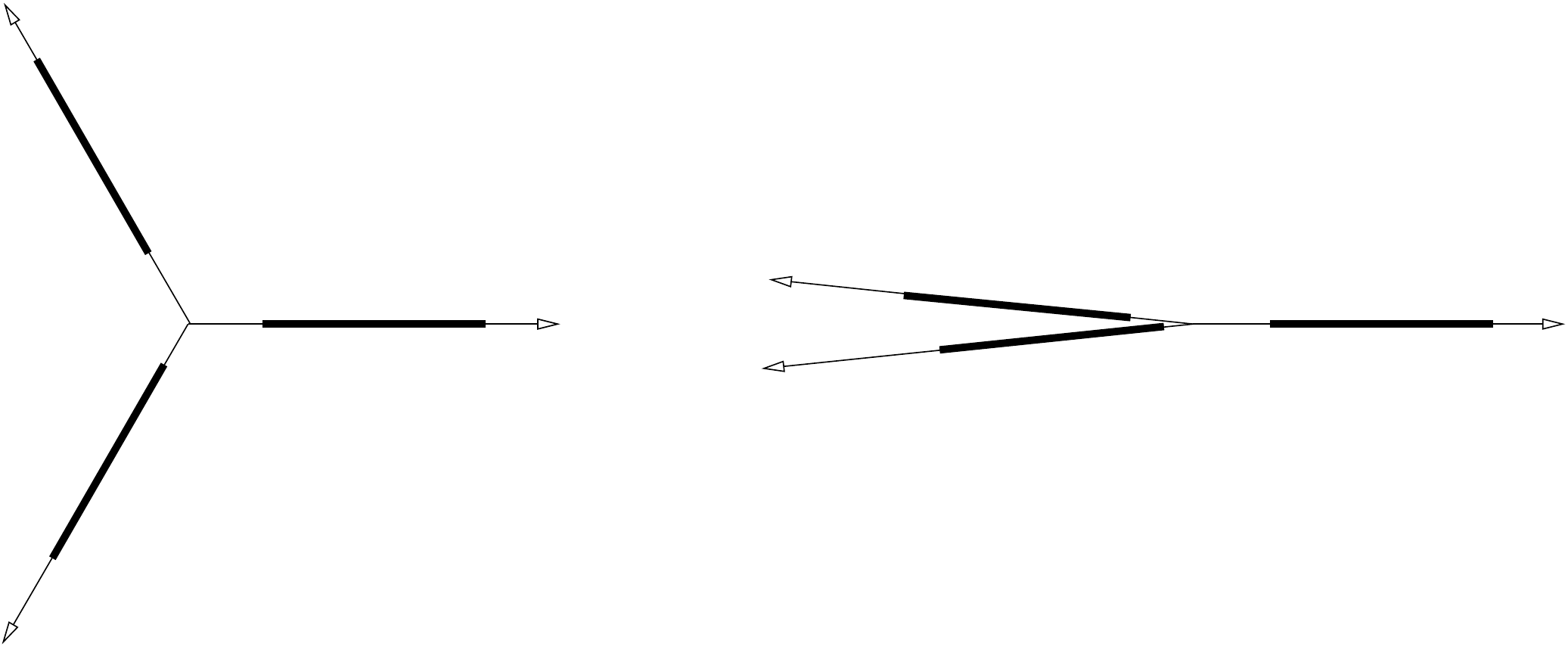}      
\caption{Recursive construction (sketch): before and after a flattening step.}
\label{fig:construction}
\end{figure}

The set $S_0$ is just one point. Suppose that we already have a construction $S_{k-1}$
with $n/3=3^{k-1}$ points. Take three rays with a common origin with angles $120^\circ$ between them.
Place a  copy of  $S_{k-1}$ along each ray.
Let $A$, $B$, $C$ denote the three copies of $S_{k-1}$, called {\em blocks} and let 
$S'_k=A \cup B \cup C$, a set of $n=3^k$ points.
We can assume that no two points of $S'_{k}$ have the same $x$-coordinate, otherwise we apply a rotation.
Let $S=S_{k}$ be a flattened copy of $S'_{k}$, that is, apply the transformation
$(x,y)\longrightarrow (x, \varepsilon y)$ to $S'_{k}$; $n=|S|=3^k$.

Let $\pi$ be any peeling sequence of $S$, and let $\pi'$ be a \emph{prefix} of it.
We say that $S$ \emph{yields} $S'$ via $\pi'$, written $S \stackrel{\pi'}{\Rightarrow} S'$ if
applying $\pi'$ to $S$ yields $S'$. 
Observe that each ray ``supports'' $n/3$ points and the following invariant is maintained:
\begin{enumerate} [(I)] \itemsep 1pt
\item \label{inv:1} Let $\pi'$ be a \emph{prefix} of $\pi$, and assume that
  $S \stackrel{\pi'}{\Rightarrow} S'$, where $S'$ still has three \emph{active} rays,
  \ie, none of $A$, $B$, or $C$, have been completely erased.
  Then $S'$ has exactly $3$ extreme vertices, one from each of $A$, $B$, $C$.
\end{enumerate}

Replace every element of $A$ (resp. $B$, $C$) by the symbol $a$ (resp. $b$, $c$).
This sequence contains $n/3$ $a$'s, $b$'s and $c$'s
and we call it the {\em simplified peeling sequence $\pi^*$}.
It just tells us from which block points are peeled off at each step.

Consider now the subsequence $\pi_A$ of $\pi$ consisting of the elements of $A$.
Observe that $\pi_A$ is a peeling sequence of $A$. Clearly the same holds for $B$ and $C$, or any other subset of $S$.

We can obtain all peeling sequences $\pi$ of $S$ in two steps: we first determine the 
simplified peeling sequence $\pi^*$ and then we expand it to a peeling sequence.
There are less than $3^n$ simplified peeling sequences. Let $\pi^*$ be a simplified peeling sequence.
Consider the last $a$, last $b$, and last $c$ in $\rho$ and take the first of them.
Assume for simplicity that it is the $a$, and it is at position $s$. Clearly, $n/3 \le s \le n$.

Estimate now the number of ways $\pi^*$ can be expanded to a peeling sequence of $S$.
Observe that $A$ is erased before $B$ and $C$, therefore, whenever an element of $A$ was peeled off,
there was exactly one element of $A$, $B$, and $C$ on the convex hull (in particular, there were only three extreme points). 
So the point that was peeled off was determined by $\pi^*$.
That is, if $\pi^*$ is fixed, then there is just one possible peeling order for the points of $A$,
namely the decreasing order of distance from the origin. The same applies to the points of $B$ and $C$
that were peeled off by the time $A$ is erased. 

By the previous observations, for the elements of $B$ (resp. $C$) we have at most
$g(B)=g(C)=g(S_{k-1}) \leq 27^{n/3} =3^n$ choices. Consequently, the induction step yields
\[ g(S_k) \leq 3^n g^2(S_{k-1}) \leq 27^n, \]
as required.
\qed

\paragraph{Proof of Corollary~\ref{cor:simple}.}
For $3^k<n<3^{k+1}$, let $P_n$ be any subset of $S_{k+1}$ of size $n$.
Then by monotonicity, we have $g(n) \le g(P_n) \le g(P_{3^{k+1}}) \le 27^{3n}=19683^n$.
\qed

\paragraph{Remark.} A construction similar to that in the proof of Theorem~\ref{thm:simple}
was used by Edelsbrunner and Welzl~\cite{EW86} to prove the existence of $n$-element point sets
with $\Omega(n \log{n})$ halving lines.

\subsection{Second construction and analysis } \label{subsec:second}

\paragraph{Preliminaries.} For $0 \le p \le 1$, denote by 
\[ H(p)=-p \log p- (1-p) \log (1-p), \] 
the \emph{binary entropy function}, where $\log$ stands for the logarithm in base 2.
(By convention, $0\log{0}=0$.)
We will use the following estimate in our calculation; see, \eg,~\cite[Lem.~3.6]{Gr11}, or~\cite[Cor.~10.3]{MU17}.
For any integer $n \ge 1$ and $0 \leq \alpha \leq 1/2$, we have
\begin{equation} \label{eq:entropy}
\binom{n}{\alpha n} \le 2^{n H(\alpha)}.
\end{equation}
Also, note that for any integer $n \geq 1$ and real $x$, where $n/3 \leq x \leq n$, we have
\begin{equation} \label{eq:fandc}
\binom{n}{\lceil x \rceil} \leq 2 \binom{n}{\lfloor x \rfloor}.
\end{equation}

\begin{lemma}\label{lem:coef}
  Let $n \geq 24$ be a positive integer. Then
  \begin{equation} \label{eq:coef}
    {n \choose k} \leq 2^n/6 \text{ for every } 0 \leq k \leq n.
  \end{equation}
\end{lemma}
\begin{proof}
  For $n=24$ the largest binomial coefficient is $\binom{24}{12}$ and one can check that
  $\binom{24}{12} \leq 2^{24}/6$. Suppose that $n>24$ and the statement holds for $n-1$.
  We may assume that $k \geq 1$ (since the inequality clearly holds for $k=0$). 
  Then $\binom{n}{k}=\binom{n-1}{k}+\binom{n-1}{k-1} \leq 2 \cdot 2^{n-1}/6=2^{n}/6$.
\end{proof}

A  key fact in the argument is the following. 

\begin{lemma}\label{lem:divide}
  Let $Z=X \cup Y$ be  a set of points in general position,
  $X \cap Y =\emptyset$, 
  $|X|=n_1$, $|Y|=n_2$, $|Z|=n_1+n_2=n$.
  where $X \cap Y =\emptyset$. 
  Then $g(Z) \leq \binom{n}{n_1} g(X) g(Y)$. 
\end{lemma}
\begin{proof}
Let $\pi$ be a peeling sequence of $Z$ and let $\pi^*$ be its
simplified peeling sequence, that is, in $\pi$
we replace every element of $x$ (resp. $Y$) by
$x$ (resp. $y$).
Clearly, there are $\binom{n}{n_1}$ simplified peeling sequences.

Let  $\pi_X$ be the subsequence of $\pi$, consisting of the elements of $X$.  
Observe that $\pi_X$ is a peeling sequence of $X$.
Define  $\pi_Y$ analogously, and clearly it is a peeling sequence of $Y$.
Therefore, at most $g(X) g(Y)$ peeling sequences can have the same
simplified peeling sequence, consequently,
$g(Z) \leq \binom{n}{n_1} g(X)g(Y)$.
\end{proof}

\paragraph{Proof of Theorem \ref{thm:main}.}
Write $a=12.29$. For every $n$ we construct the sets $S_n$ recursively, such that $S_n$ contains $n$ points and 
$g(S_n) \le a^n/100$ for $n \geq 3$. Moreover, $S_n$ is {\em flat}, just like in the previous construction.
The set $S_1$ is just one point; $S_2$ is a point pair; and $S_3$ is a flat (obtuse) triangle. 

Suppose that $n \geq 4$ and for every $i<n$ we have already constructed $S_i$ satisfying the requirements.
In order to do a better recursion, we choose a specific variant in the previous construction.
Let $n=n_1+n_2+n_3$, where $n_i=\lfloor n/3\rfloor$ or $\lceil n/3\rceil$.
Take three rays $r_1$, $r_2$, $r_3$ with a common origin $O$ and
with angles $120^\circ$ between them, such that $r_1$ is horizontal.
Place a  copy of  $S_{n_1}$ on $r_1$ and call it $B_1$, 
a  copy of  $S_{n_2}$ on $r_2$, close to $O$, and call it $B_2$, 
and a  copy of  $S_{n_3}$ on $r_3$, far from $O$, and call it $B_3$.
Finally, let $S_n$ be a flattened copy of this set of $n$ points.
For simplicity we still denote  the three corresponding components of
$S_n$ by $B_1$, $B_2$, $B_3$, respectively.
Observe, that the projections
of $B_1$, $B_2$, and $B_3$ are separated on the $x$-axis, they come in this order,
and in $S_n$ all points are very close to their projections.
We have thereby defined the set $S_n$ for every $n$ and clearly it has $n$ points. 

\smallskip
We prove that $g(S_n) \le a^n/100$ by induction on $n$. The induction basis is $3 \leq n \leq 35$:
Obviously $g(P) \leq |P|!$ for every point-set $P$. By the special structure of $S_n$, we have
$g(S_n) \leq 3^{\lfloor n/3 \rfloor} (\lceil 2n/3 \rceil)! \leq a^n/100$ for every $3 \leq n \leq 35$.
Suppose now that $n \geq 36$ and that $g(S_i) \leq a^i/100$ for every $i$, where $3 \leq i \leq n-1$.
Let $n=n_1+n_2+n_3$, where $n_i=\lfloor n/3\rfloor$ or $\lceil n/3\rceil$.
Let $B_1$, $B_2$, $B_3$, be its three blocks, $B_i$ is an affine copy of $S_{n_i}$.
For any peeling sequence $\pi$ of $S_n$, we define its  simplified peeling sequence $\pi^*$
so that we replace every element of $B_{i}$ by $i$. 

Now let  $\pi$ be a peeling sequence of $S_n$ and let $\pi^*$ be its  simplified peeling sequence.
Just like in the proof of Theorem \ref{thm:simple},
take the last $1$, last $2$, and last $3$ in $\pi^*$,
and assume that the first of them is at position $s$.
Since one of the blocks is peeled off in step $s$, we have $\lfloor n/3\rfloor \leq s \leq n$.

Observe that until peeling step $s+1$, all three rays were active, so 
there was exactly one element of $B_1$, $B_2$, and $B_3$ on the convex hull.
Therefore, the first $s$ elements of $\pi$ are determined by $\pi^*$.
We distinguish four cases based on the value of $s$.
We estimate now the number of corresponding simplified peeling sequences, 
and the number of ways they can be expanded to a peeling sequence of $S_n$.
We denote the resulting upper bound estimates by $\sigma_1,\sigma_2,\sigma_3,\sigma_4$, and then show
that $\sum_{j=1}^4 \sigma_j \leq a^n/100$.
More precisely, we will prove that
\[ \sigma_1 \leq \frac12 \cdot \frac{a^n}{100}, \ \ 
\sigma_2 \leq \frac16 \cdot \frac{a^n}{100}, \ \ 
\sigma_3 \leq \frac16 \cdot \frac{a^n}{100}, \ \ 
\sigma_4 \leq \frac16 \cdot \frac{a^n}{100}.
\]

 \medskip
 \emph{Case 1.}
 $\lfloor n/3\rfloor \leq s \leq \lceil 5n/9 \rceil$. 
 Let $\sigma_1$ denote the number of corresponding peeling sequences of $S_n$.
 Let $\T$ be the corresponding set of simplified peeling sequences.
 By~\eqref{eq:entropy}, \eqref{eq:fandc}, and~\eqref{eq:coef},  we have
 \begin{align*}
 |\T| &\leq 3{\lceil 5n/9 \rceil \choose \lfloor n/3 \rfloor} {\lceil 2n/3 \rceil \choose \lfloor n/3\rfloor}
 = 3{\lceil 5n/9 \rceil \choose \lceil 5n/9 \rceil - \lfloor n/3 \rfloor} {\lceil 2n/3 \rceil \choose \lfloor n/3\rfloor}\\
 &\leq 3 \cdot 2 \cdot {\lceil 5n/9 \rceil \choose \lfloor 2n/9 \rfloor} 2^{\lceil 2n/3\rceil}/6 
 \leq 2^{\lceil 5n/9 \rceil \cdot H(0.4)}  2^{\lceil 2n/3\rceil} \\
&\leq 2^{(5 H(0.4) + 6) n/9}  \cdot 2^{8/9}  \cdot 2^{2/3}
 \leq 2^{10.855n/9} \cdot 2^{8/9}  \cdot 2^{2/3} \\
 &= 2^{14/9} \cdot 2^{10.855n/9}. 
\end{align*}

Let $\pi^*\in\T$.
Suppose for simplicity, that at position $s$ there is a $1$. So block $B_1$ is finished
first, and its peeling order is determined. For the other two blocks we have at most
$g(S_{n_2})g(S_{n_3})$ possibilities. We have $n_1, n_2, n_3\le \lceil n/3\rceil$, therefore,
by the induction hypothesis and Lemma \ref{lem:divide} we have
 \begin{align*}
   \sigma_1 &\leq |\T| \cdot g^2(S_{\lceil n/3\rceil}) \leq 2^{14/9} \cdot 2^{10.855n/9} \cdot
   \frac{a^{n -\lfloor  n/3 \rfloor}}{10000}
   \leq \frac12 \cdot \frac{a^n}{100},
\end{align*}
 where the last inequality follows from the two inequalities:
 \[ 2^{23/9} \cdot a^{2/3} \leq 100,  
\text{ and } 2^{10.855/3} \leq a. \]

\medskip
\emph{Case 2.}
$\lceil 5n/9\rceil \leq s \leq \lfloor 2n/3\rfloor$.
Let $\sigma_2$ denote the number of corresponding peeling sequences of
 $S_n$. Let $\T$ be the corresponding set of simplified
peeling sequences.
By Lemma~\ref{lem:coef} we have
\[ |\T| \leq 3{\lfloor 2n/3\rfloor \choose \lfloor n/3\rfloor }{\lceil 2n/3\rceil \choose \lceil n/3\rceil} \leq 2^{4n/3}/12. \]
Let $\pi^*\in\T$.
Suppose again there is an $1$ at position $s$, so block $B_1$ is finished first, and its peeling order is determined.
Moreover, since  $\lceil 5n/9\rceil \leq s$,
at least  $\lceil n/9\rceil$ points of $B_2$ or $B_3$ have already been removed, say, from $B_2$.
These removed points were the extreme points of $B_2$ in one direction, that is, a sub-block of $B_2$ has been removed.
The remainder of $B_2$ after step $s$ of the peeling process is a subset of (an affine image of) 
two sub-blocks of $S_{n_2}$, and by construction, both have size at most $\lceil n/9\rceil$.
By Lemma \ref{lem:divide} this set has at most $2^{2\lceil n/9\rceil}g(S_{\lceil n/9\rceil})^2$ peeling sequences.
For $B_3$ we have at most $g(S_{\lceil n/3\rceil})$ possibilities. Note that $n/9 \geq 4$ and so the induction hypothesis
applies. Therefore, 
\begin{align*}
\sigma_2 &\leq |\T| \cdot g(S_{\lceil n/3\rceil}) \cdot g^2(S_{\lceil n/9\rceil}) \cdot 2^{\lceil 2n/9\rceil} \\
&\leq \frac{2^{12n/9}}{12} \cdot \frac{a^{\lceil n/3 \rceil}}{100}  \cdot \frac{a^{\lceil n/9 \rceil}}{100}
\cdot \frac{a^{\lceil n/9 \rceil}}{100}
\cdot 2^{2n/9} \cdot 2^{8/9} \\
&\leq \frac16 \cdot 2^{14n/9} \cdot a^{5n/9} \cdot a^{22/9} \cdot \frac{1}{10^6} \leq \frac16  \cdot \frac{a^n}{100}.
\end{align*}
The last inequality follows from the two inequalities:
 \[ a^{22/9} \leq 10^4,  \text{ and } 2^{7/2} \leq a. \]

\medskip
\emph{Case 3.}
$\lceil 2n/3 \rceil \leq s \leq \lfloor 7n/9\rfloor$.
Let $\sigma_3$ denote the number of corresponding peeling sequences and 
$\T$ the corresponding set of simplified peeling sequences.
We have
\[ |\T| \leq 3{\lfloor 7n/9 \rfloor \choose \lceil n/3\rceil}{\lceil 2n/3\rceil \choose \lceil n/3 \rceil} \leq 2^{13n/9}/12. \]
Let $\pi^*\in\T$.
Suppose again that at position $s$ there is an $1$. So block $B_1$ is finished first, and its peeling order is determined.
Moreover, since $\lceil 2n/3\rceil \leq s$, at least $\lceil n/3 \rceil$ additional points have been removed,
that is, two sub-blocks of $B_2$ or $B_3$, of at least $\lfloor n/9 \rfloor$ points,
either two sub-blocks from one of them or one sub-block from each of them. 
We can argue similarly to the previous cases and get the following estimate.
\begin{align*}
\sigma_3 &\leq |\T| \left(g(S_{\lceil n/3\rceil}) \cdot g(S_{\lceil n/9\rceil}) +g^4(S_{\lceil n/9\rceil}) \cdot 2^{4\lceil n/9\rceil}\right) \\
&\leq \frac{2^{13n/9}}{12} \left( \frac{a^{\lceil n/3 \rceil} a^{\lceil n/9 \rceil}}{10^4} +
\frac{a^{4\lceil n/9 \rceil} 2^{4\lceil n/9 \rceil}}{10^8} \right) \\
&\leq  \frac{2^{13n/9}}{12} \left( \frac{a^{4n/9} \cdot a^{14/9}}{10^4} +   
\frac{(2a)^{4n/9} \cdot (2a)^{32/9}}{10^8} \right) \\
&\leq  \frac16  \cdot \frac{a^n}{100}.
\end{align*}
The last inequality follows from the three inequalities:
\[ a^{14/9} \leq 10^2/2, \ \ (2a)^{32/9} \leq 10^6/2, \ \ 
\text{ and } 2^{17/5} \leq a. \]

\medskip
\emph{Case 4.} $\lceil 7n/9 \rceil \leq s \leq n$.
Let $\sigma_4$ denote the number of corresponding peeling sequences and 
$\T$ the corresponding set of simplified peeling sequences. Obviously $|\T| \leq 3^n$.
Let $\pi^*\in\T$ and suppose again that block $B_1$ is finished first, so its peeling order is determined.
Since $\lceil 7n/9\rceil \leq s$, three sub-blocks of $B_2$ or $B_3$ have been removed. 
The following estimate is implied.
\begin{align*}
  \sigma_4 &\leq |\T| \cdot g^3(S_{\lceil n/9\rceil}) \cdot 2^{2\lceil n/9\rceil} \\
  &\leq 3^n \cdot 2^2 \cdot a^3 \cdot 2^{2n/9} \cdot a^{n/3}/100 \\
  &\leq \frac16  \cdot \frac{a^n}{100}.
\end{align*}
The last inequality is clearly satisfied. 

\smallskip
Finally, $g(S_n)=\sigma_1+\sigma_2+\sigma_3+\sigma_4 \leq a^n/100$, concluding the induction step
and thereby also the proof of Theorem~\ref{thm:main}.
\qed

\section{Higher dimensions} \label{sec:d-dim}

In this section we prove Theorem~\ref{thm:d-dim} and its corollary.

\paragraph{Lower bound.}
Let $P$ be any $n$-element point set in general position. By the assumption, $\conv(P)$ has
at least $d+1$ extreme vertices, removal of each yields a set of $n-1$ points in general position.
Any two peeling sequences resulting from the removal of two different extreme vertices are clearly different.
As such, $g(\cdot)$ satisfies the recurrence $g(n) \geq (d+1) \cdot g(n-1)$. Consequently,
$g(n) = \Omega((d+1)^n)$.

\paragraph{Upper bound.}
We proceed similarly to the planar case.
We construct the sets $S_k$ recursively for $k \in \NN$,  
such that the set $S_k$ contains $n=(d+1)^k$ points and $g(S_k)\le (d+1)^{(d+1)n}$.
Moreover, $S_k$ is {\em thin}, that is, all points are very close
(compared to the minimum distance in $S_k$) to a line.

The set $S_0$ is just one point. Suppose that we already have a construction $S_{k-1}$
with $n/(d+1)=(d+1)^{k-1}$ points. Take a regular simplex $\Delta \subset \RR^d$ centered at the 
origin and such that $(1,0,\ldots,0)$ is a vertex of $\Delta$.
Take the $(d+1)$ rays $r_1,\ldots,r_{d+1}$ from the origin to its vertices.
Place a  copy of  $S_{k-1}$ along each ray.
Let $B_1, \ldots, B_{d+1}$ denote the copies of $S_{k-1}$, and let 
$S'_k=\cup_{i=1}^{d+1}B_i$, a set of $n=(d+1)^k$ points.

We can assume that no two points of $S'_{k}$ have the same $x_1$-coordinate, otherwise we apply a rotation.
Let $S=S_{k}$ be a flattened copy of $S'_{k}$, that is, apply the transformation
$(x_1, x_2, \ldots, x_d)\longrightarrow (x_1, \varepsilon x_2, \ldots, \varepsilon x_d)$
to $S'_{k}$.

Let $\pi$ be any peeling sequence of $S$.
Observe that each ray ``supports''  $n/(d+1)$ points and the following invariant is maintained
%(due to the fact that the origin remains in the interior of the convex hull for as long as the current set
%still has $d+1$ active rays):
%
\begin{enumerate}[(I)$_d$] \itemsep 1pt
\item \label{inv:d-dim} Let $\pi'$ be a \emph{prefix} of $\pi$, and assume that
  $S \stackrel{\pi'}{\Rightarrow} S'$, where $S'$ still has $d+1$ active rays,
  \ie, none of its blocks have been completely erased.
  Then $S'$ has exactly $d+1$ extreme vertices, one from each block.
\end{enumerate}

We obtain the simplified peeling sequence as before, replace every element of
$B_i$ by $i$. To obtain all peeling sequences $\pi$ of $S$, first we determine the 
simplified peeling sequence $\pi^*$ and then we expand it to a peeling sequence.
There are less than $(d+1)^n$ simplified peeling sequences. Let $\pi^*$ be a simplified peeling sequence.
Consider the last $i$ for every $1\le i\le d+1$ 
and take the first of them.
Assume for simplicity that it is $1$, and it is at position $s$. Clearly, $n/d \le s \le n$.

Since $B_1$ was finished first, the peeling order of its points is determined. 
For the elements of $B_i$, $2\le i\le d+1$ we have at most
$g(B_i)=g(S_{k-1}) \leq (d+1)^{n}$ choices. Consequently, the induction step yields
\[ g(S_k) \leq (d+1)^n g^d(S_{k-1}) \leq (d+1)^{(d+1)n}, \]
as required.
\qed

\paragraph{Proof of Corollary~\ref{cor:d-dim}.}
For $(d+1)^k<n<(d+1)^{k+1}$, let $P_n$ be any subset of $S_{k+1}$ of size $n$.
Then by monotonicity, we have
\[ g(n) \le g(P_n) \le g(P_{(d+1)^{k+1}}) \le (d+1)^{(d+1)^2n}, \]
as claimed.
\qed

\section{Concluding remarks} \label{sec:remarks}

Observe that the convex layer decomposition of a point set $P$ --- mentioned in Section~\ref{sec:intro} --- yields
a set of peeling sequences naturally derived from it: remove the points from one layer, one by one, before moving
to the next layer. Indeed, this is so, since any point of the layer under removal is still extreme at that step.   
In general, if there are $m$ layers and their sizes are $h_1,h_2,\ldots,h_m$ counting from outside,
where $h_1,\ldots,h_{m-1} \geq 3$, $h_m \geq 1$, and $\sum_{i=1}^m h_i =n$, 
then there are $h_1! h_2! \cdots h_m!$  peeling sequences given by the convex layer decomposition.

Apart from the case of points in convex position, the set of peeling sequences corresponding to
layer by layer removal of the points is a strict subset of the set of peeling sequences of $P$.
It is worth noting that this subset can be much smaller than the whole set. For example,
consider a point set with $n/3$ layers, where each layer is a triangle (and so the point set
is the vertex set of $n/3$ nested triangles). Then there are $6^{n/3} = 1.817\ldots^n$ 
peeling sequences given by the convex layer decomposition, whereas the total number of
peeling sequences is $\Omega(3^n)$. 

Our estimates on the growth rate of $g(n)$ are now closer, but a substantial gap remains.
A natural question is whether the trivial bound of $\Omega(3^n)$ can be improved.

\begin{problem} \label{problem:4^n}
  Is there a constant $\delta >0$ such that $g(n)=\Omega((3+\delta)^n)$? 
\end{problem}

\end{document}